\documentclass{amsart}
\usepackage{amsmath,amssymb,amsthm}
\usepackage{graphicx,epsfig}
\usepackage{enumerate}
\usepackage{xypic}
\usepackage{tikz}
\usepackage{pgf}
\usepackage{comment}
\usepackage{mathrsfs}
\usetikzlibrary{arrows}
\usepackage{csquotes}
\usepackage{multicol}
\usepackage{hyperref}
\usetikzlibrary{calc}
  \input{xy}
  \xyoption{all}
\usepackage{url}
\usepackage[]{algorithm2e}
\usepackage{setspace}

\usepackage{comment}
\numberwithin{equation}{section}
\theoremstyle{plain}
\newtheorem{thm}{Theorem}[section]
\newtheorem{lem}[thm]{Lemma}
\newtheorem{dfn}[thm]{Definition}
\newtheorem{cor}[thm]{Corollary}
\newtheorem{nota}[thm]{Notation}

\newtheorem{prop}[thm]{Proposition}

\newtheorem{rem}[thm]{Remark}
\newtheorem{ex}[thm]{Example}
\newtheorem{headtheorem}{Theorem}

\newtheorem{headexample}{Example}

\DeclareMathOperator{\reg}{reg}
\DeclareMathOperator{\proj}{proj}

\DeclareMathOperator{\spec}{spec}
\DeclareMathOperator{\tor}{Tor}

\DeclareMathOperator{\Ext}{Ext}

\DeclareMathOperator{\Ima}{Im}

\DeclareMathOperator{\Z}{\mathbb{Z}}
\DeclareMathOperator{\f}{\mathbf{f}}
\DeclareMathOperator{\KO}{\textbf{K}}
\DeclareMathOperator{\K}{K}

\DeclareMathOperator{\CE}{\textbf{C}}

\DeclareMathOperator{\N}{\mathbb{N}}

\DeclareMathOperator{\Supp}{ Supp}
\DeclareMathOperator{\Spec}{ Spec}
\DeclareMathOperator{\codim}{ codim}

\DeclareMathOperator{\Coker}{Coker}
\DeclareMathOperator{\HF}{HF}
\DeclareMathOperator{\Jac}{Jac}

\DeclareMathOperator{\mm}{\mathfrak{m}}

\begin{document}
\title[Multigraded regularity of Complete Intersections ]{Multigraded regularity of complete intersections}
\author{Marc Chardin}
\address{Institut de math\'ematiques de Jussieu, CNRS \& Sorbonne Universit\'e, 4 place Jussieu, 75005 Paris , France}
\email{marc.chardin@imj-prg.fr}

\author{Navid Nemati}
\address{ Universit\'{e} C\^{o}te d'Azur, Inria, 2004 route des Lucioles, 06902 Sophia Antipolis, France}
\email{navid.nemati@inria.fr}

\begin{abstract}
$V$ is a complete intersection scheme in a multiprojective space if it can be defined by an ideal $I$ with as many generators as $\codim(V)$. We investigate the multigraded regularity of complete intersections scheme in $\mathbb{P}^n\times \mathbb{P}^m$. We explicitly compute many values of the Hilbert functions of $0$-dimensional complete intersections. We show that these values only depend upon $n,m$, and the bidegrees of the generators of $I$.  As a result, we provide a sharp upper bound for the multigraded regularity of $0$-dimensional complete intersections. 
\end{abstract}
\maketitle
\section{Introduction}
\noindent The theory of syzygies offers a microscope for looking at systems of equations. Castelnuovo-Mumford regularity is an important invariant in commutative algebra and algebraic geometry which is strongly related to syzygies. D. Eisenbud and S. Goto in \cite{EisenbudGoto} showed that Castelnuovo-Mumford regularity can be obtained from the minimal free resolution. The minimal free resolution was first introduced by Hilbert to study Hilbert function. It is one of the finest invariant that we can associate with a finitely generated graded module M over a polynomial ring. 
Castelnuovo-Mumford regularity measures the 
maximum degree of the syzygies and provides a quantitative version of Serre vanishing 
theorem for the associated sheaf. In particular, it  bounds the largest degree of the 
minimal generators and the smallest twist for which the sheaf is generated by its global 
sections. It has been used as a measure for the complexity of computational problems in 
algebraic geometry and commutative algebra. The two most frequent definitions of $
\mathbb{Z}$-graded Castelnuovo-Mumford regularity are the one in terms of graded 
Betti numbers and the one using local cohomology.

An extension of Castelnuovo-Mumford regularity for the multigraded setting in a special case  was first introduced by Hoffman and Wang in \cite{Hoffman-Wang}.  Later by Maclagan and Smith in \cite{Maclagen-Smith}, and Botbol and Chardin in \cite{Botbol-Chardin} in a more general 
setting.   The main motivation for studying regularity over multigraded polynomial rings 
was from toric geometry. Maclagan and Smith \cite{Maclagen-Smith} developed a multigraded theory of regularity of sheaves on a simplicial toric variety $\mathcal{X}$ with an algebraic variant defined in terms of the vanishing of graded pieces of $H^i_{B}(M)$, the $i$-th local cohomology module of $M$. Here $B$ is the irrelevant ideal of the homogeneous  coordinate ring of $\mathcal{X}$. 

The conceptual difficulty of computing multigraded regularity lies in the simple fact that bounded subsets of $\mathbb{Z}^r$ need not have single maximal or minimal elements. This makes it very hard to capture the vanishing or non-vanishing of the multigraded pieces of local cohomology.  
One can define a complete intersection in projective space as a subscheme defined by as many forms as its codimension. In this case, it corresponds to complete intersection homogeneous ideals; therefore many homological invariants are determined by the degrees of the forms.  One natural question that may arise is to ask the same question for the complete intersection schemes in a product of projective spaces.

Proposition 6.7 in \cite{Maclagen-Smith} implies that  the multigraded graded regularity a $0$-dimensional schemes is the same as the stabilization region of the Hilbert functions. Understanding the Hilbert functions and the minimal free resolutions of the coordinate rings of points in multiprojective space is included among the list of open problems in commutative algebra found in the survey article of Peeva-Stillman\cite{Peeva-Stillman2}.

Based on the above motivations, our main concern in this article is to study multigraded regularity and bigraded Hilbert function complete intersection points in $\mathbb{P}^n\times \mathbb{P}^m$. 
In Section \ref{sec-bigraded}, we start by setting our notation and define multigraded regularity. We show that if the ring is bigraded, then  one can consider an alternative way to define the multigraded regularity (see Theorem \ref{Bregularity}).

In section \ref{sec-CI}, we study the multigraded regularity of general complete intersection schemes in $\mathbb{P}^n\times \mathbb{P}^m$. A scheme $V\subset \mathbb{P}^n\times \mathbb{P}^m$ is complete intersection if it can be defined by as many forms as its codimension. Our motivation is to show which homological invariants are determined by the degrees of these forms.

In Section \ref{sec:points}, we  focus on the case of complete intersection points in $\mathbb{P}^n\times \mathbb{P}^m$. 
The main result of this section is Theorem \ref{Hilbert points}. For a complete intersection scheme of point $V\subseteq \mathbb{P}^n\times \mathbb{P}^m
$ defined by $f_1,\dots, f_{n+m}$, many values of its bigraded Hilbert function are independent from the choices of $f_i$'s and they only depend upon their bidegrees. Moreover, in some regions, there is a duality among these values. 
\begin{headexample}(Example \ref{example})
Let $S=k[x_0,x_1,x_2, y_0,y_1,y_2]$, and $I=(f_1,\dots,f_4)$ is generated by $4$ bigraded forms of bidegree $(2,2)$ such that $\proj(S/I)=V$ is complete intersection scheme of points . For any $\mu \notin (2,6)+(-\N,\N)\cup (6,2)+ (\N,-N) $, the Hilbert function of $V$ at $\mu$ is independent from the choices of $f_i$'s and can be computed explicitly via Theorem  \ref{Hilbert points}.
\end{headexample}
In Section \ref{sec:generic},  we study generic complete intersection points in $\mathbb{P}^n\times \mathbb{P}^m$.  The main Theorem of this section is the following
\begin{headtheorem}(Theorem \ref{thm:generic})
Let $S=k[x_0,\dots,x_n,y_0,\dots,y_m]$ be a bigraded polynomial ring over a field $k$ of characteristics zero. If $I$ is generated by $n+m$ generic forms of bidegree $(d,e)$,  then the scheme $V$ defined by $I$ is a set of reduced points and 
$$
| \lbrace \mu\in \mathbb{N}^2 | \HF_V(\mu)\neq \deg(V)\rbrace| < \infty,
$$
which means the natural projections are one-to-one. 
\end{headtheorem}

\section{Multigraded regularity in bigraded setting}\label{sec-bigraded}

\begin{nota}\label{notation}
From now on we use the following notation in the rest of the paper. Let $S=k[x_0,\dots,x_n,y_0,\dots,y_m]$ be a bigraded polynomial ring where $\deg(x_i)=(1,0)$ and $\deg(y_i)=(0,1)$. Let  $B_1= (x_0,\dots,x_n)$, and $ B_2=(y_0,\dots,y_m)$ be the irrelevant ideal of $\mathbb{P}^n$ and $\mathbb{P}^m$. Define  $ B:=B_1\cdot B_2$ the irrelevant ideals of $\mathbb{P}^n\times \mathbb{P}^m$, and $\mm= B_1+B_2$ the maximal ideal of $S$ as a standard graded ring.
\end{nota}

We state the definition of multigraded regularity in the case of bigraded rings; for further reading, we refer readers to \cite{Botbol-Chardin, Maclagen-Smith}.
\begin{dfn}
Let $M$ be a graded $S$-module. The support of the module $M$ is 
$$
\Supp(M):= \lbrace \gamma \vert M_{\gamma}\neq 0\rbrace.
$$
\end{dfn}
\begin{dfn}(Multigraded regularity)
Let $C\subseteq \mm$ be a finitely generated bigraded ideal of $S$. Define for a  
$\gamma\in \Z^2$, $M$ is \textbf{ $(C,\gamma)$-regular} if 
$$
\gamma + \N^2 \cap \cup_i \Supp(H^i_C(M))+\mathcal{F}_{i-1}= \emptyset,
$$
where $\mathcal{F}_{i}:= \lbrace (i-1,0), (i-2,1),\dots, (0,i-1)\rbrace$ for $i>0$ and $\mathcal{F}_0:= \lbrace0\rbrace$, $\mathcal{F}_{-1}= -\mathcal{F}_1$, and $\mathcal{F}_i=0$ for $i<-1$.  $C$-regularity of $M$ is
$$
\reg_C(M) := \lbrace \gamma\in \Z^r \mid \text{M is } (C,\gamma)\text{-regular}\rbrace.
$$
In particular, if $C=B$, $B$-regularity of $M$, $\reg_B(M)$, is called the multigraded regularity of $M$.
\end{dfn}
Our approach for computing multigraded regularity relies on Mayer-Vietoris sequence to relate the various cohomology modules.

\begin{rem}\label{Mayer-Vietoris-exactseq}
Let $S = k[x_1,\dots,x_m,y_1,\dots,y_n]$ be a bigraded polynomial ring. Define  $B_1,B_2,B$ and $\mm$ as in Notation \ref{notation}. Then we have the following complex between local cohomology modules:
$$
\cdots \rightarrow H^{i-1}_B(M)\rightarrow H^i_{\mm}(M)\rightarrow H^i_{B_1}(M)\oplus H^i_{B_2}(M)\rightarrow H^i_{B}(M)\rightarrow \cdots 
$$
\end{rem}
The following Theorem provide an alternative way to define multigraded regularity  in the bigraded setting by knowing $\mm, B_1$, and $B_2$-regularity.
\begin{thm}\label{Bregularity}
Let $S = k[x_1,\dots,x_m,y_1,\dots,y_n]$ be a bigraded polynomial ring and  $M$ is a  graded $S$-module. Adopt  Notations \ref{notation}, then
$$
\reg_B(M) = \reg_{\mm}(M)\cap \reg_{B_1}(M)\cap \reg_{B_2}(M).
$$
\end{thm}
\begin{proof}
By analyzing the spectral sequences correspond to \v{C}ech-Koszul  double complex $C^{\bullet}_{B}
K_{\bullet}(\underline{\textbf{x}},\underline{\textbf{y}};M)$ we get
$$
\Supp(\tor_i^S(M,k))\subseteq \cup_{j=0}^{i} \Supp(H^j_{B}(M))+ \mathcal{E}_{i+j}
$$
where $\mathcal{E}_{i}$ denote the set of twists of the summands in the $i$-th step of the minimal free resolution of $B$; see \cite[Section 4.2]{Botbol-Chardin} for  more details regarding the relation between $B$-regularity and the support of Tor modules. 
By \cite[Corollary 3.12]{Botbol-Chardin}, for every $i$ we have
\begin{align*}
\Supp(H^i_{\mm}(M))&\subseteq \cup \Supp(\tor_{m+n-i}^S(M,k))+ \Supp(H^{m+n}_{\mm}(S))\\
 & \subseteq \cup \Supp (\tor_{m+n-i}^S(M,k))-(m,n)-\N^2\\
 &\subseteq \cup \cup_{j=0}^{i} \Supp(H^j_{B}(M))+ \mathcal{E}_{m+n-(i-j)}-(m,n)-\N^2\\
 &\subseteq \cup_{j=0}^{i} \Supp(H^j_{B}(M))+ \mathcal{E}_{i-i}-\N^2.
\end{align*}
Therefore
$$
\Supp(H^i_{\mm}(M))+\mathcal{F}_{i-1}\subseteq \cup_{j=0}^{i} \Supp(H^j_{B}(M))+\mathcal{F}_{j-1}-\N^2,
$$
which means 
\begin{equation}\label{regB}
\reg_B(M)\subseteq \reg_{\mm}(M).
\end{equation}

Suppose $\mu\in \reg_B(M)$. By the definition,  $\left(\mu- \mathcal{F}_{i-1}+\N^2\right)\cap \Supp H^i_B(M)= \emptyset$ for every $i$. By \ref{regB}, $\left(\mu- \mathcal{F}_{i-1}+\N^2\right)\cap \Supp H^i_{\mm}(M)=\emptyset$. By using Mayer-Vietoris exact sequence one gets $\left(\mu- \mathcal{F}_{i-1}+\N^2\right)\cap \left(\Supp H^i_{B_1}(M)\cap \Supp H^i_{B_2}(M)\right)=\emptyset$, that implies 
$$
\reg_B(M)\subseteq \reg_{\mm}(M)\cap \reg_{B_1}(M)\cap \reg_{B_2}(M).
$$
Suppose $\mu \in \reg_{\mm}(M)\cap \reg_{B_1}(M)\cap \reg_{B_2}(M)$. By the definition, for every $i$, $\left(\mu- \mathcal{F}_{i-1}+\N^2\right)\cap \left(\Supp H^i_{B_1}(M)\cap \Supp H^i_{B_2}(M)\right)=\emptyset$, and $\left(\mu-\mathcal{F}_i+\N^2\right)\cap \linebreak \Supp H^{i+1}_{\mm}(M)=\emptyset$. Since $ \left( \mu-\mathcal{F}_{i-1}+\N^2\right) \subset \left(\mu-\mathcal{F}_i+\N^2\right)$,  $\left(\mu-\mathcal{F}_{i-1}+\N^2\right)\cap \Supp H^{i+1}_{\mm}(M)= \emptyset$. The assertion follows from Mayer-Vietoris exact sequence and the definition of $\reg_B(M)$.

\end{proof}

\section{Complete intersections in $\mathbb{P}^n\times \mathbb{P}^m$}\label{sec-CI}
We start this section by stating the definition of complete intersection scheme in a product of projective that we are using in this article.

\begin{dfn}
 A subscheme $V\subseteq \mathbb{P}^n\times \mathbb{P}^m$, is a \textit{complete intersection} if $V=\proj (S/I)$ where $I$ is generated by $\codim(V)$ bihomogeneous elements.
\end{dfn}
Let $V$ be a complete intersection scheme in a projective space, then there exist ideal $I$ defining $V$ with as many generators as the codimension of $V$. In this case, ideal $I$ is complete intersection, therefore many homological invariants are entirely determined by the degrees of the generators. Let $V= \proj(S/I)\subset \mathbb{P}^n\times \mathbb{P}^m$ be a complete intersection scheme, where $I$ is generated by $\codim(V)$ forms. In this case, $I$ is not complete intersection unless $n,m=1$ in which case $I$ is not the defining ideal of $V$ (it is not $B$-saturated). Although the similar argument would not work in this case, but our goal is to show that many homological invariants are only depend upon the degrees of the generators of $I$.

\begin{prop}\cite[Example 8.4.2]{Fult}\label{degree}
Let $S=k[x_0,\dots,x_n,y_0,\dots,y_m]$ be a bigraded polynomial ring where $\deg(x_i)=(1,0)$ and $\deg(y_i)=(0,1)$. Let $I=(f_1,\dots,f_{n+m})$ generated by $n+m$ forms of degree $(d_i,e_i)$ and $V$ be the scheme defined by $I$. If $I$ is a complete intersection, then
$$
\deg(V)= \sum d_{i_1}\cdots d_{i_n}\cdot e_{j_1}\cdots e_{j_m},
$$ 
where the sum is over all permutations $(d_{i_1},\dots, d_{i_n}, e_{j_1},\dots, e_{j_m})$ of $(1,\dots,n+m)$ with $i_1<i_2 <\cdots < i_n$ and $j_1<j_2<\cdots < j_m$. In particular if $d_i=d$ and $e_i=e$ for all $i$, then 
$$
\deg(V)=\binom{n+m}{n}d^ne^m.
$$
\end{prop}

\begin{prop}
Let $I=(f_1,\dots,f_r)$ be a bigraded ideal defining a complete 
intersection of codimension $r$ and let $\deg(f_i)= \textbf{d}_i
$. For $\mu \in \cap_i \reg_B(H_i(\KO(\f,S))$, 
$$
\HF_{S/I}(\mu)= P(\mu),
$$
where $P$ is a polynomial that only depends upon $\textbf{d}_1,\dots, \textbf{d}_r$.
\end{prop}
\begin{proof}
By the Serre Grothendieck formula \cite[Proposition 4.27]{Botbol-Chardin}
\begin{equation}\label{SGF}
HP_{S/I}(\mu)= \HF_{S/I}(\mu) + \sum_i (-1)^i \dim H^i_{B}(S/I)_\mu.
\end{equation}
Hence, if $\mu\in \reg_B(S/I)$ then $HP_{S/I}(\mu)= \HF_{S/I}(\mu)$.  Set  
\begin{equation}\label{definition chi}
\chi(\mu):= \sum_i (-1)^i \dim (H_i(\KO(\f,S))_\mu= \sum_i (-1)^i \dim(K_i)_\mu.
\end{equation}
Note that the second equality shows that $\chi$ is a function that only depends upon the degrees. 
Since $I$ is complete intersection, $H_i(\KO(\f,S)= H^0_B(H_i(\KO(\f,S))$. Hence, $\Supp(H_i(\KO(\f,S))\cap \reg_B(H_i(\KO(\f,S))=\emptyset $. Therefore,   for $\mu \in \cap_{i>0} \reg_B(H_i(\KO(\f,S))$,  $
\chi(\mu)= \HF_{S/I}(\mu).
$
Hence, for $\mu \in \cap_{i} \reg_B(H_i(\KO(\f,S))$, 
$$
HP_{S/I}(\mu)= \HF_{S/I}(\mu)= \chi(\mu).
$$
Setting $P:= HP_{S/I}$ the first equality shows that $\HF_{S/I}(\mu)= P(\mu)$ and the second equality shows this function only depends upon the degrees.
\end{proof}
\begin{prop}\label{Mayer-Vietoris}
With the Notation \ref{notation}, for any bigraded free $S$-module $M$, 
$$
H^{n+m+1}_B(M)\cong H^{n+m+2}_{\mm}(M).
$$
 Furthermore, If $n=m$ then 
$$
H^{n+1}_{B}(M)\cong H^{n+1}_{B_1}(M)\oplus H^{n+1}_{B_2}(M),
$$
else if $n<m$ then, 
$$
H^{n+1}_{B}(M)\cong H^{n+1}_{B_1}(M)\quad \text{and} \quad H^{m+1}_{B}(M)\cong H^{m+1}_{B_2}(M).
$$
\end{prop}
\begin{proof}
Note that $H^{a}_{B_1}(M)=H^{b}_{B_2}(M)=H^{c}_{\mm}(M)=0$ if $a\neq n+1$, $b\neq m+1$ and $c\neq n+m+2$. The Mayer-Vietoris exact sequence
$$
\cdots \rightarrow H^i_{\mm}(M)\rightarrow H^i_{B_1}(M)\oplus H^i_{B_2}(M)\rightarrow H^i_{B}(M)\rightarrow  H^{i+1}_{\mm}(M)\rightarrow \cdots
$$
gives the results.
\end{proof}

\begin{thm}\label{main B regularity}
Let $S=k[x_0,\dots,x_n,y_0,\dots,y_m]$ be a bigraded polynomial ring where $\deg(x_i)=(1,0)$ and $\deg(y_i)=(0,1)$. Define $B=(x_0,\dots,x_n)\cdot (y_0,\dots,y_m)$ the irrelevant ideal of $\mathbb{P}^n\times \mathbb{P}^m$. Assume $I=(f_1,\dots,f_{r})$ with  $ \deg(f_i)= (d_i,e_i)$ and $V$ be the scheme  defined by $I$. If $\codim(V)=r$ then
$$\begin{array}{lll}
\bullet &\Supp(H_i(\KO(\f,S))&= \Supp(V_{n+1+i})\cup  \Supp(W_{m+1+i}) \quad \text{for}\,\, i>0\\
\bullet &\Supp (H^i_B(S/I))&= \Supp(V_{n+1-i})\cup \Supp(W_{m+1-i}) \, \text{for} \,\, 0\leq i\leq \dim V,\\
\bullet &\Supp(H^{\dim V+1}_B(S/I))&= \lbrace (\mu,\nu) \vert H^{\dim V}(V, \mathcal{O}_V(\mu,\nu))\neq 0\rbrace\\
&&\subseteq  \Supp(V_{r-m})\cup \Supp(W_{r-n})\cup \{\sigma- \Supp\left(S/I \right)\},
\end{array}$$
where $V_i$ (resp. $W_i$) is a subquotient of $H^{n+1}
_{B_1}(\K_i(\f,S))$ (resp. $H^{m+1}_{B_2}(\K_i(\f,S))$), and $\sigma := \sum_i(d_i,e_i)-(n+1,m+1)$.
\end{thm}
\begin{proof}
Consider the double complex $\CE^{\bullet}_{B}(\KO_{\bullet}(\f,S))$ and suppose $n\leq m$. If we 
start  taking homologies vertically, by \ref{Mayer-Vietoris} in the third page we have:
$$\begin{array}{ccccc}
0&0&\cdots&0&0\\
\vdots&\vdots&\vdots&\vdots&\vdots\\
V_{r}& V_{r-1}&\cdots &V_1&V_0\\
\vdots&\vdots&\vdots&\vdots&\vdots\\
W_{r}& W_{r-1}&\cdots &W_1&W_0\\
\vdots&\vdots&\vdots&\vdots&\vdots\\
M& \cdots&0&0&0, 
\end{array}$$
where 
\begin{align*}
M&\cong H_r(H^{n+m+1}_B(\KO(\f,S))\\
&\cong H_r\left( H^{n+m+2}_{\mm}\left(\KO(\f,S)\right)\right)\\
&\cong H_r
\left(( \KO\left(\f,S\right)\left(\sigma)\right)^{\star}\right)\\
&\cong (S/I)^{\star} \left(\sigma\right).
\end{align*}
If we start taking homology horizontally, in the second page we have 
$$
\begin{array}{cccccc}
\cdots & 0  &   H_{\gamma} & \cdots & H_1  &    
H^0_B(S/I)\\
\cdots&0&0&0&0&H^1_B(S/I)\\
\cdots&0& 0&0&0&\vdots\\
\vdots &\vdots &\vdots&\vdots&\vdots&H^{\dim ( V)+1}_B(S/I)\\
\vdots &\vdots &\vdots&\vdots&\vdots& 0,
\end{array}
$$
where $\gamma := \max \{ 0, r-n-1, r-m-1\} $.
The result is obtained by comparing the two abutments.
\end{proof}
\begin{dfn}
Let $\textbf{d}=(d,e)\in \mathbb{N}^2$ and $r\in \mathbb{N}$, define
\begin{align*}
v_i &:= i\cdot \mathbf{d}-(n+1,0),\\
w_i&:= i\cdot \mathbf{d}-(0,m+1),\\
\sigma&:= r\cdot \mathbf{d}-(n+1,m+1).
\end{align*}
\end{dfn}
\begin{lem}\label{lem supp homologies}
Let $I=(f_1,\dots,f_r)\subseteq S$ and $\deg(f_i)=\textbf{d}= (d,e)$.  Suppose $V$ is the scheme defined by $I$,  if $\codim(V)=r$ then
\begin{align*}
\Supp(H_i(\KO(\f,S))\subseteq &  v_{n+1+i}+(-\mathbb{N}, \mathbb{N})\cup  w_{m+1+i}+(\mathbb{N},-\mathbb{N}) \quad \text{for}\,\, i>0,\\
\Supp (H^i_B(S/I))\subseteq& v_{n+1-i}+(-\mathbb{N}, \mathbb{N})\cup w_{m+1-i}+(\mathbb{N},-\mathbb{N}) \quad \text{for} \,\, 0\leq i\leq \dim V.
\end{align*}
In addition, the inclusions are sharp.
\end{lem}
\begin{proof}
With the proof of Theorem \ref{main B regularity}, in the second page of the spectral sequence we have for $p< n+m+1$
$$
^{v}E^{\infty}_{p,q}=^{v}E^2_{p,q}= \begin{cases}
 H_q(H^{n+1}_{B_1}\KO(\f,S)) \quad \mbox{if} \,\,p=n+1,\\
 H_q(H^{n+1}_{B_2}\KO(\f,S)) \quad \mbox{if}\,\, p=m+1,\\
 0 \quad \mbox {else}.
\end{cases}
$$
Combining with the Theorem \ref{main B regularity} shows the inclusions. The sharpness follows from the fact that  for any $q$, $v_q $ (resp. $w_q$) is in the support of $H^{n+1}_{B_1}K_q(\f,S)$ (resp. $H^{m+1}_{B_2}K_q(\f,S)$) and it is not in the support of $ H^{n+1}_{B_1}K_{q-1}(\f,S)$ and $H^{n+1}_{B_1}K_{q+1}(\f,S)$ (resp. $H^{m+1}_{B_2}K_{q-1}(\f,S)$ and $H^{m+1}_{B_2}K_{q-1}(\f,S)$).
\end{proof}
\begin{cor} \label{cor equality H1}
Let $I=(f_1,\dots,f_r)\subseteq S$ and $\deg(f_i)=\textbf{d}= (d,e)$ and $V$ be the scheme defined by $I$.  If $\codim(V)=r$, then
\begin{itemize}
\item[$(1)$ ] if  $\mu\notin v_{i}+(-\mathbb{N},\mathbb{N})\cup w_{j}+(\mathbb{N}, -\mathbb{N})$ for $i=r-m,r-m-1$ and $j=r-n,r-n-1$ then
$$
H^{\dim(V)+1}_B(S/I))_\mu= (S/I)^{\star}_{\sigma-\mu}. 
$$
\end{itemize} 
 \end{cor}
\begin{proof}
Note that$\Supp  \left(H^{n+1}_{B_1}(K_{i}(\f,S))\right)=v_i +(-\mathbb{N}, \mathbb{N}) $ and $\Supp H^{m+1}_{B_2}(K_{i}(\f,S))= w_i +(\mathbb{N}, +\mathbb{N})$ for all $i$.   Consider the double complex $\CE^{\bullet}_{B}(\KO_{\bullet}(\f,S))$ and suppose $n\leq m$ as in the proof of Theorem \ref{main B regularity}.  Since $\Supp(V_i)\cap \Supp(W_j)=\emptyset$, there will be no nonzero map in the spectral sequence among them. In this case, in the  $n+2$-th page we have an induced map $\psi_1: M \rightarrow W_{r-n-1}$, and in the $m+2$-th page we have  an induced map $\psi_2: \ker \psi_1\rightarrow V_{r-m-1}$.  If $\mu \notin v_{r-m}+(-\mathbb{N},\mathbb{N})\cup w_{r-n}+(\mathbb{N}, -\mathbb{N})$, $H^r_B(S/I)_{\mu}\cong (\ker \psi_2)_\mu$. In addition, if $\mu \notin v_{r-m-1}+(-\mathbb{N},\mathbb{N})\cup w_{r-n-1}+(\mathbb{N}, -\mathbb{N})$
$$
 H^{\dim(V)+1}_B(S/I)_{\mu}\cong (\ker \psi_2)_\mu= M_\mu= (S/I)^{\star}_{\sigma-\mu}.
$$
\end{proof}
\section{Complete intersection points in $\mathbb{P}^n\times \mathbb{P}^m$}\label{sec:points}
\noindent In this section we apply the results in the previous section to the case of complete intersection points in $\mathbb{P}^n\times \mathbb{P}^m$. In the rest of this section, $I=(f_1,\dots,f_{n+m})$ and $V$ be the complete intersection  scheme  of points defined by $I$ with $\textbf{d}:= \deg(f_i)=(d,e)$.
\begin{dfn}
Set  $\Gamma_i := \Supp(H^{n+1}_{B_1}(\K_{n+i}(\f,S)))\cup \Supp(H^{m+1}_{B_2}(\K_{m+i}(\f,S)))
$.
Notice that $\Gamma_i=\emptyset$ if and only if $i\geq \max \{n,m\}$.  If $i\leq m$ then 
$$
\Supp(H^{n+1}_{B_1}(\K_{n+i}(\f,S))) = v_{n+i}+ (-\mathbb{N}, \mathbb{N}),
$$
and similarly for $\Supp(H^{m+1}_{B_2}(\K_{m+i}(\f,S)))$ if $i\leq n$.
\end{dfn}
\begin{dfn}\label{F'}
For a function $F: \mathbb{Z}^2\rightarrow \mathbb{Z}$ define
\begin{align*}
F^{\star}(a,b)&:= F(-a,-b),\\
F'(a,b)&:= F(-a,b), \,\, \text{and}\\
F''(a,b)&:= F(a,-b).
\end{align*}
\end{dfn}
\begin{thm}\label{Hilbert points}
Let $S=k[x_0,\dots,x_n,y_0,\dots,y_m]$ be a bigraded polynomial ring where $\deg(x_i)=(1,0)$ and $\deg(y_i)=(0,1)$.  Assume $V$ be a complete intersection scheme of points defined by $I=(f_1,\dots,f_{n+m})$ with $\deg(f_i)=(d,e)$. 
\begin{itemize}
\item[ $(1)$] If $\mu\notin \cup_{i\geq 2}\Gamma_i$ then
$$
\HF_{S/I}(\mu)= \chi(\mu)
$$
where $\chi((a,b))= \sum_i (-1)^i\binom{n+m}{i}\binom{n+a-id}{n}\binom{m+b-ie}{m}$.
\item[$(2)$] If $\mu\notin \Gamma_1$ then
$$
\HF_{S/I}(\mu)= \HF_V(\mu)
$$
\item[$(3)$]
If $\mu\notin \Gamma_0$ then 
$$
\HF_{V}(\mu)+ \HF_V(\sigma-\mu)= \deg(V)=\binom{n+m}{n}d^ne^m.
$$
In particular, if $\mu \in (nd-n,(n+m)e-m)+\N^2\cup ((n+m)d-n,me-m)+\N^2$ then $\HF_V(\mu)= \deg(V)$ and $\mu\in \reg_B(V)$.
\item[$(4)$]
If $\mu \in \Gamma_0\setminus \{ \Gamma_{-1}\cup \Gamma_1\}$  then
$$
\HF_V(\mu) +\HF_V(\sigma-\mu)= \binom{n+m}{n}(d^ne^m-\epsilon(\mu)),
$$
where $\epsilon(\mu)= \HF_S'(\mu-v_n)+ \HF_S''(\mu-w_n)$ as in the Definition \ref{F'}.
\item[$(5)$]
If $\mu\in \cup_{i\geq 1} \Gamma_i \setminus \{\Gamma_0\}$ then
$$
\HF_{V}(\mu)= \binom{n+m}{n}d^ne^m- \chi(\sigma-\mu).
$$
 In particular, if $\mu\in \Gamma_1\setminus \{\Gamma_0\cup \Gamma_2\}$ then
$$
\dim(I_V/I)_\mu= \binom{n+m}{n+1}\HF_S'(\mu-v_{n+1})+ \binom{n+m}{m+1}\HF_S''(\mu- w_{m+1}).
$$
\end{itemize}

\end{thm}
\begin{proof}
Denote the saturation of $I$ with respect to $B$ by $J$. In this case $\HF_V(\mu)= \dim(S/J)_{\mu}$ for all $\mu\in \mathbb{Z}^2$. In the proof we use these two  simple fact that $\mu \in \Gamma_{i-1}\cup \Gamma_{i+1}$ yields $\mu \in \Gamma_i$  and if $\mu\in \Gamma_i$ then $\sigma-\mu\in \Gamma_{-i}$. By Proposition \ref{degree} the Hilbert polynomial of $V$ is $D:= \binom{n+m}{n}d^ne^m$ which in this case is equal to the $\deg(V)$. By Serre Grothendieck formula,
\begin{equation}\label{SGF2}
\HF_{S/J}(\mu)+\dim (H^1_B(S/J)_{\mu})=  D
\end{equation}
\\
$(1)$ If $\mu\notin \cup_{i\geq 2}\Gamma_i$, by 
Lemma \ref{lem supp homologies} $(H_i(\K(\f,S))_
\mu=0$ for all $i\geq 1$.  Therefore by Equation 
\ref{definition chi}, $(S/I)_\mu=\chi(\mu)$.\\
$(2)$ If $\mu \notin \Gamma_1$, by Lemma \ref{lem supp homologies}, $(H^0_B(S/I))_\mu=0$ therefore $(S/I)_\mu= (S/J)_\mu$.\\
$(3)$ First note that $\mu, \sigma-\mu \notin \Gamma_0$. We claim that $\mu\notin \Gamma_{-1}$ or $\sigma-\mu \notin \Gamma_{-1}$. Suppose not, then $\mu\in \Gamma_{-1}\cap \Gamma_{1}$ which yields $\mu\in \Gamma_0$ which is a contradiction. So assume $\mu\notin \Gamma_{-1}$. By  Corollary \ref{cor equality H1}, 
$$
\dim (H^1_B(S/I)_{\mu})=  \HF_{S/I}(\sigma-\mu).
$$
 Since $\sigma-\mu\notin \Gamma_1$, by part $(2)$ and Equation \ref{SGF}  
$$
\binom{n+m}{n}d^ne^m-  \HF_V(\mu) \HF_{V}(\sigma-\mu).
 $$
 The same argument works if we assume $\sigma-\mu \notin \Gamma_{-1}$. \\
 $(4)$ By part $(1)$ $\HF_{S/I}(\mu)= \HF_V(\mu)$. By considering the double complex as in the proof of Theorem \ref{main B regularity} and Lemma  \ref{lem supp homologies}, 
 \begin{align*}
  &\left(H_n(H^{n+1}_{B_1}\K(\f,S))\right)\mu= (H^{n+1}_{B_1}(K_{n}(\f,S))_{\mu},  \,\,  \text {and} \\
& \left(H_m(H^{m+1}_{B_2}\K(\f,S))\right)_\mu= \left(H^{m+1}_{B_2}(K_{m}(\f,S)\right)_{\mu}.
 \end{align*}
  By the abutment of the spectral sequence 
$$
\left( H^1_B(S/J)\right)_\mu= \HF_{S/I}(\sigma-\mu)
+ \dim (H^{n+1}_{B_1}(K_{n}(\f,S))_{\mu}+.\left(H^{m+1}_{B_2}(K_{m}(\f,S)\right)_{\mu}.
$$
Since $\sigma-\mu\notin \Gamma_1$, by parts $(1)$ and $(2)$, $\HF_{S/I}(\mu)= \HF_{V}(\mu)$ and by the definition, $ \dim (H^{n+1}_{B_1}(K_{n}(\f,S))_{\mu}+ \left(H^{m+1}_{B_2}(K_{m}(\f,S)\right)_{\mu}= \binom{n+m}{n}\epsilon(\mu)$. The assertion follows by $ \left( H^1_B(S/J)\right)_\mu =  \binom{n+m}{n}d^ne^m- \HF_V(\mu) $.\\
$(5)$ By part $(3)$, $\HF_{V}(\mu)= \binom{n+m}{n}d^ne^m-\HF_{V}(\sigma-\mu)$. Because $\mu\notin \Gamma_{-1}$ ($\mu \in \Gamma_{-1}$ yields $\mu \in \Gamma_0$) therefore $\sigma-\mu\notin \Gamma_1$, which  by part $(2)$, yields $\HF_{V}(\sigma -\mu)= \HF_{S/I}(\sigma -\mu)$. On the other hand, since $\mu \notin \cup_{i\geq 2}\Gamma_i$, by part $(1)$, $\HF_{S/I}(\sigma-\mu)= \chi(\sigma-\mu)$.
 \end{proof}

Let $S=[x_0,x_1,x_2,y_0,y_1,y_2]$ and $\deg(x_i)=(1,0)$ and $\deg(y_i)= (0,1)$. Let $I=(f_1,\dots,f_4)$ where $\deg(f_i)= (2,2)$. Suppose scheme $V$ defined by $I$ is complete intersection. The following picture  demonstrate the regions $\Gamma_i$ in the Theorem \ref{Hilbert points}.
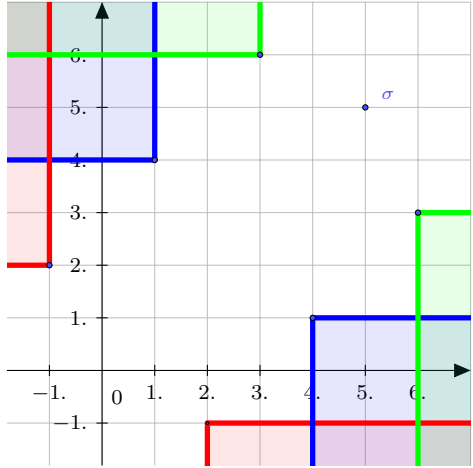
\begin{center}
\definecolor{qqffqq}{rgb}{0.,1.,0.}
\definecolor{ffqqqq}{rgb}{1.,0.,0.}
\definecolor{uuuuuu}{rgb}{0.26666666666666666,0.26666666666666666,0.26666666666666666}
\definecolor{ffffqq}{rgb}{0.,0.,1.}
\definecolor{ududff}{rgb}{0.30196078431372547,0.30196078431372547,1.}
\definecolor{cqcqcq}{rgb}{0.7529411764705882,0.7529411764705882,0.7529411764705882}
\begin{tikzpicture}[line cap=round,line join=round,>=triangle 45,x=0.7cm,y=0.7cm]
\draw [color=cqcqcq,, xstep=0.7cm,ystep=0.7cm] (-1.8,-1.8) grid (7.,7.);
\draw[->,color=black] (-1.8,0.) -- (7.,0.);
\foreach \x in {-1.,1.,2.,3.,4.,5.,6.}
\draw[shift={(\x,0)},color=black] (0pt,2pt) -- (0pt,-2pt) node[below] {\footnotesize $\x$};
\draw[->,color=black] (0.,-1.8) -- (0.,7.);
\foreach \y in {-1.,1.,2.,3.,4.,5.,6.}
\draw[shift={(0,\y)},color=black] (2pt,0pt) -- (-2pt,0pt) node[left] {\footnotesize $\y$};
\draw[color=black] (0pt,-10pt) node[right] {\footnotesize $0$};
\clip(-1.8,-1.8) rectangle (7.,7.);
\fill[line width=2.pt,color=ffffqq,fill=ffffqq,fill opacity=0.10000000149011612] (1.,4.) -- (1.0025348000000005,10.376970399999998) -- (-5.374435599999996,10.379505199999999) -- (-5.376970399999998,4.002534800000001) -- cycle;
\fill[line width=2.pt,color=ffqqqq,fill=ffqqqq,fill opacity=0.10000000149011612] (-1.,2.) -- (-1.,10.) -- (-9.,10.) -- (-9.,2.) -- cycle;
\fill[line width=2.pt,color=ffqqqq,fill=ffqqqq,fill opacity=0.10000000149011612] (2.,-1.) -- (2.,-8.) -- (9.,-8.) -- (9.,-1.) -- cycle;
\fill[line width=2.pt,color=ffffqq,fill=ffffqq,fill opacity=0.10000000149011612] (4.,1.) -- (4.,-7.) -- (12.,-7.) -- (12.,1.) -- cycle;
\fill[line width=2.pt,color=qqffqq,fill=qqffqq,fill opacity=0.10000000149011612] (3.,6.) -- (3.,15.7382) -- (-6.7382,15.7382) -- (-6.7382,6.) -- cycle;
\fill[line width=2.pt,color=qqffqq,fill=qqffqq,fill opacity=0.10000000149011612] (6.,3.) -- (6.,-10.) -- (19.,-10.) -- (19.,3.) -- cycle;
\draw [line width=2.pt,color=ffffqq] (1.,4.)-- (1.0025348000000005,10.376970399999998);
\draw [line width=2.pt,color=ffffqq] (1.0025348000000005,10.376970399999998)-- (-5.374435599999996,10.379505199999999);
\draw [line width=2.pt,color=ffffqq] (-5.374435599999996,10.379505199999999)-- (-5.376970399999998,4.002534800000001);
\draw [line width=2.pt,color=ffffqq] (-5.376970399999998,4.002534800000001)-- (1.,4.);
\draw [line width=2.pt,color=ffqqqq] (-1.,2.)-- (-1.,10.);
\draw [line width=2.pt,color=ffqqqq] (-1.,10.)-- (-9.,10.);
\draw [line width=2.pt,color=ffqqqq] (-9.,10.)-- (-9.,2.);
\draw [line width=2.pt,color=ffqqqq] (-9.,2.)-- (-1.,2.);
\draw [line width=2.pt,color=ffqqqq] (2.,-1.)-- (2.,-8.);
\draw [line width=2.pt,color=ffqqqq] (2.,-8.)-- (9.,-8.);
\draw [line width=2.pt,color=ffqqqq] (9.,-8.)-- (9.,-1.);
\draw [line width=2.pt,color=ffqqqq] (9.,-1.)-- (2.,-1.);
\draw [line width=2.pt,color=ffffqq] (4.,1.)-- (4.,-7.);
\draw [line width=2.pt,color=ffffqq] (4.,-7.)-- (12.,-7.);
\draw [line width=2.pt,color=ffffqq] (12.,-7.)-- (12.,1.);
\draw [line width=2.pt,color=ffffqq] (12.,1.)-- (4.,1.);
\draw [line width=2.pt,color=qqffqq] (3.,6.)-- (3.,15.7382);
\draw [line width=2.pt,color=qqffqq] (3.,15.7382)-- (-6.7382,15.7382);
\draw [line width=2.pt,color=qqffqq] (-6.7382,15.7382)-- (-6.7382,6.);
\draw [line width=2.pt,color=qqffqq] (-6.7382,6.)-- (3.,6.);
\draw [line width=2.pt,color=qqffqq] (6.,3.)-- (6.,-10.);
\draw [line width=2.pt,color=qqffqq] (6.,-10.)-- (19.,-10.);
\draw [line width=2.pt,color=qqffqq] (19.,-10.)-- (19.,3.);
\draw [line width=2.pt,color=qqffqq] (19.,3.)-- (6.,3.);
\begin{scriptsize}
\draw [fill=ududff] (5.,5.) circle (1.0pt);
\draw[color=ududff] (5.429146698908396,5.239946296842116) node {$\sigma$};
\draw [fill=ududff] (-1.,2.) circle (1.0pt);
\draw [fill=ududff] (2.,-1.) circle (0.5pt);
\draw [fill=ududff] (1.,4.) circle (1.0pt);
\draw [fill=ududff] (4.,1.) circle (1.0pt);
\draw [fill=ududff] (3.,6.) circle (1.0pt);
\draw [fill=ududff] (6.,3.) circle (1.0pt);
\draw [fill=ududff] (1.0025348000000005,10.376970399999998) circle (2.5pt);
\draw [fill=uuuuuu] (-5.374435599999996,10.379505199999999) circle (2.5pt);
\draw [fill=uuuuuu] (-5.376970399999998,4.002534800000001) circle (2.5pt);
\draw [fill=ududff] (-1.,10.) circle (2.5pt);
\draw [fill=uuuuuu] (-9.,10.) circle (2.5pt);
\draw [fill=uuuuuu] (-9.,2.) circle (2.5pt);
\draw[color=uuuuuu] (-2.571072346548095,2.493602445416766) node {$M$};
\draw [fill=ududff] (2.,-8.) circle (0.5pt);
\draw [fill=uuuuuu] (9.,-8.) circle (0.5pt);
\draw [fill=uuuuuu] (9.,-1.) circle (0.5pt);
\draw [fill=ududff] (4.,-7.) circle (0.5pt);
\draw [fill=uuuuuu] (12.,-7.) circle (2.5pt);
\draw [fill=uuuuuu] (12.,1.) circle (0.5pt);
\draw [fill=ududff] (3.,15.7382) circle (2.5pt);
\draw [fill=uuuuuu] (-6.7382,15.7382) circle (2.5pt);
\draw [fill=uuuuuu] (-6.7382,6.) circle (2.5pt);
\draw [fill=ududff] (6.,-10.) circle (2.5pt);
\draw [fill=uuuuuu] (19.,-10.) circle (2.5pt);
\draw [fill=uuuuuu] (19.,3.) circle (2.5pt);
\end{scriptsize}
\end{tikzpicture}
\end{center}
$\Gamma_{-1}$ is the red region, $\Gamma_0$ is the blue region and $\Gamma_1$ is the green region.
By Theorem \ref{Hilbert points}, Hilbert function of $V$ at the bidegrees except the intersection of blue and green   are independent from the choices of $f_i$'s and they can be computed via $\chi$ and $\epsilon$ defined in the Theorem \ref{Hilbert points}. On the other hand, the rest do depends upon the $f_i$'s. Here we computed two example via computer system Macaulay2 \cite{M2}. 
\begin{ex}\label{example}
Let $S=k[x_0,x_1,x_2, y_0,y_1,y_2]$, $I = (x_0^2y_0^2,x_1^2y_1^2,x_2^2y_2^2,(x_0+x_1+x_2)^2(y_0+y_1+y_2)^2)$ and $V$ be the complete intersection scheme of points defined by $I$. For $(0,0)\leq \mu\leq (7,7)$ the bigraded Hilbert function  $\HF_V(\mu)$ is
$$
\begin{matrix}
 \textcolor{orange}{24}& \textcolor{orange}{72}& \textcolor{green}{96}& \textcolor{green}{96}& 96& 96& 96& 96 \\
 \textcolor{orange}{24}& \textcolor{orange}{72}& \textcolor{green}{96}& \textcolor{green}{96}& 96& 96& 96 &96 \\
 \textcolor{blue}{21} &\textcolor{blue}{63} &86 &90 &93 &\textcolor{red}{95} &96 &96 \\
 \textcolor{blue}{15} &\textcolor{blue}{45}& 66& 78& 87& 93& 96& 96 \\
10 &30 &48 &64 &78 &90 &\textcolor{green}{96} &\textcolor{green}{96} \\
6 & 18& 32& 48 &66 &86 &\textcolor{green}{96} &\textcolor{green}{96} \\
 3  &9  &18 &30 &\textcolor{blue}{45} &\textcolor{blue}{63} &\textcolor{orange}{72}& \textcolor{orange}{72} \\
 1 & 3&  6&  10& \textcolor{blue}{15} &\textcolor{blue}{21} &\textcolor{orange}{24} &\textcolor{orange}{24} 
 \end{matrix}.
 $$
Where blue  corresponds to the bidegrees in $\Gamma_0$, green corresponds to $\Gamma_1$ and orange to their intersections. Also red indicate the $\HF_V(\sigma)$. For $\mu\notin \Gamma_0$ by Theorem \ref{Hilbert points} parts $(1),(2)$ and $(3)$ one can compute the bigraded Hilbert function and for blue points by part $(4)$. For the orange points, Theorem \ref{Hilbert points} does not say anything. In addition, any points except the orange ones only depend on the degree of the generator of $I$, which in this case is $(2,2)$.

Asking   \texttt{random}
command in Macaulay2 and constructing $I'$ with $4$ generators of bidegree $(2,2)$. Let  $V'$ be the complete intersection scheme of points defined by $I$. For $(0,0)\leq \mu\leq (13,7)$  the bigraded Hilbert function $\HF_{V'}(\mu)$ is
$$
\xymatrixrowsep{0.4em}
\xymatrixcolsep{0.4em}
\xymatrix{
33&84& 96& 96& 96& 96& 96& 96 & 96 &96&96&96&96&96\\
27& 81& 96& 96 &96 &96 &96 &96 &96&96&96&96&96&96\\
21& 63& 86& 90& 93& 95& 96& 96&96&96&96&96&96&96 \\
15& 45& 66& 78& 87& 93& 96& 96 &96&96&96&96&96&96\\
10 &30 &48& 64 &78 &90 &96 &96&96&96&96&96&96&96 \\
6  &18& 32& 48& 66& 86& 96& 96&96&96&96&96&96&96 \\
3  &9 &18 &30 &45 &63 &81 &84 &96&96&96&96&96&96 \\
1  &3 & 6 & 10& 15 &21 &28 &36 &45&55&66&78&91&96
}
$$

The only differences are in the orange spots. In addition, in this case $\HF_{V'}(0,13)= \HF_{V'}(13,0)=96$ which means the natural projection of $V'$ to each $\mathbb{P}^n$ and $\mathbb{P}^m$ is one to one.
\end{ex}

\begin{thm}
Let $S=k[x_0,\dots,x_n,y_0,\dots,y_m]$ be a bigraded polynomial ring where $\deg(x_i)=(1,0)$ and $\deg(y_i)=(0,1)$.  Assume $V$ be a complete intersection scheme of points defined by $I=(f_1,\dots,f_{n+m})$ with $\deg(f_i)=(d,e)$. Let  $\pi_1: V\rightarrow \mathbb{P}^n$ and $\pi_2: V\rightarrow \mathbb{P}^m$ be the natural projections. Then
\begin{align*}
\max \lbrace \reg (\pi_1^{-1}(p)) \, \, \text{for} \,\, p\in \mathbb{P}^n\rbrace&= \min \lbrace a\in \mathbb{N} | (a,b)\in \reg_B(V) \,\, \text{for some }\, b\rbrace\\
 \max \lbrace \reg (\pi_2^{-1}(p)) \, \, \text{for} \,\, p\in \mathbb{P}^m\rbrace &= \min \lbrace b\in \mathbb{N} |   (a,b)\in \reg_B(V) \,\, \text{for some }\, a\rbrace.
\end{align*}
\end{thm}
\begin{proof}
Adopt Notation \ref{notation} and let $J$ be the 
saturation of $I$ with respect to $B$. Since $S/J$ is $B
$-saturated, it is also $B_1$ and $B_2$ saturated. Let $
\mu \gg 0$ and $\nu \gg 0$,  therefore $H^1_{\mm}(S/
J)_{\mu,\nu}= H^2_{\mm}(S/J)_{\mu,\nu}=0$.  Suppose  $\mu \gg 0$, Hence by Mayer-Vietoris exact sequence
$$
H^1_{B_1}(S/J)_{\mu,\nu}\oplus H^1_{B_2}(S/J)_{\mu,\nu}= H^1_B(S/J)_{\mu,\nu}.
$$
Since $S/J$ is $B_1$-saturated, for $\mu\gg 0$, $H^1_{B_1}(S/J)_{\mu,\nu}=0$ which means
$$
 H^1_{B_2}(S/J)_{\mu,\nu}= H^1_B(S/J)_{\mu,\nu}.
$$

We state the proof for the first equality, the argument 
for the second one is the same. There exist $(a_1,b_1), 
(a_2,b_2)\in \mathbb{N}^2$ such that $H^1_{\mm}
(S/J)_{\mu}= H^2_{\mm}(S/J)_\nu=0$ for all $\mu
\geq (a_1,b_1)$ and $\nu\geq (a_2,b_2)$. Also for any 
$a\in \mathbb{N}$, there exists $b_a$ such that  for all 
$\mu\geq b_a$, $(H^1_{B_2}(S/J))_{(a,\mu)}=0$. Set 
$$
b:= \max_{a\in\mathbb{N}} \lbrace b_1, b_2, b_a
\rbrace.
$$  
The Mayer-Vietoris short exact sequence yields 
$H^1_{B}(S/J)_{(\mu,\nu)}= H^1_{B_1}(S/J)_{(\mu,\nu)}$ if $\nu> b$. Therefore
\begin{align*}
&\min \lbrace a\in \mathbb{N} | (a,b)\in \reg_B(V) \,\, \text{for some }\, b\rbrace= \min \lbrace\mu \in \mathbb{N}| H^1_{B_1}(S/J)_{\mu,\nu}=0 \,\, \text{for}\,\, \nu>b \rbrace\\
&= \min \lbrace \mu \in \mathbb{N} | \exists \, \, \mathfrak{p}\neq (y_0,y_1,y_2)\, ;\,  H^1_{B_1}(S/J\otimes_{k[y_0,y_1,y_2]}k[y_0,y_1,y_2]_{\mathfrak{p}})_{\mu,\nu}=0 \,\, \text{for }\,\, \nu>b\rbrace
\end{align*}

\end{proof}

\section{Generic complete intersection points in $\mathbb{P}^n\times \mathbb{P}^m$}\label{sec:generic}
\noindent In this subsection we adopt the Notation \ref{notation} as well. We focus on the case where ideal $I$ is generated by the generic forms. In this regard, we first need to provide an analogous version of Bertini's theorem in our case.
\begin{prop}\label{prop:Bertini}
Let $k$ be a field of characteristics zero and  $V\subset \mathbb{P}_k^n\times \mathbb{P}_k^m$ be a 
reduced scheme and $(d,e)\in \mathbb{N}^2$ with $d,e\neq 
0$. Let $f_{U}$ be a form of bidegree $(d,e)$ with 
indeterminate  coefficients $U_{\alpha,\beta}$. There exists a 
non empty open set $\Omega \subset \Spec(k[U_{\alpha,
\beta}])$ such that for any $\mathfrak{p}\in \Omega$, the 
corresponding form $f_{\mathfrak{p}}$ is such that 
$$
V\cap Z(f_{\mathfrak{p}})\subseteq \mathbb{P}^n_{k_\mathfrak{p}}\times \mathbb{P}^m_{k_\mathfrak{p}},
$$ 
where $k_{\mathfrak{p}}$ is the residue field of $k[U_{\alpha,\beta}]_{\mathfrak{p}}$, is reduced of dimension equal $\dim V-1$, unless the $\dim V=0$ in which case $V\cap Z(f_\mathfrak{p})=\emptyset$.
\end{prop}
\begin{proof}
Consider the Segre-Veronese map $\Psi: \mathbb{P}^n\times \mathbb{P}^m\rightarrow \mathbb{P}^{N}$ with $N= \binom{n+d}{d}\times \binom{m+e}{e}-1$. Let  $V'= \Psi(V)$. Under the map $\Psi$, $f_U$ is mapped to a linear form $\ell_{U}$ in  $N+1$ variables and this correspondence is one to one. By Bertini theorem (see \cite[Corollary 3.4.9]{Flenner-Ocarroll-Vogel}) there exists $\Omega$ as claimed such that for $\mathfrak{p}\in \Omega$, $V'\cap Z(\ell_\mathfrak{p})\cong V\cap Z(f_\mathfrak{p})$ is a reduced scheme with the asserted dimension.  
\end{proof}
\begin{dfn}
Forms of bidegree $(d,e)$ are in  one to one correspondence with  closed points in $\mathbb{A}_k^{\dim_k S_{(d,e)}}$. 
\end{dfn}
\begin{rem}
Notice that  $\Omega\subset \spec k[U_{\alpha,\beta}]\setminus Z(F)$ for some form $F\neq 0$. In particular, for $c= (c_{\alpha,\beta})\in k^N$ such that $F(c)\neq 0$, Proposition \ref{prop:Bertini} with $\mathfrak{p}$ generated by the elements $u_{\alpha,\beta}-c_{\alpha,\beta}$, shows that $Z(f_c)\cap V$  is reduced.
\end{rem}
\begin{thm}\label{projection generic}
Let $S=k[x_0,\dots,x_n,y_0,\dots,y_m]$ be a bigraded polynomial ring over a field $k$ of characteristic zero where $\deg(x_i)=(1,0)$ and $\deg(y_i)=(0,1)$.  Let $(d_i,e_i)\geq (1,1)$ for $i=1,\dots,n+m$ be bidegrees.
Then, there exists a non-empty open set 
$$
\Omega\subset \mathbb{A}_k^N:= \prod_{i=1}^{n+m}\mathbb{A}_k^{|S_{(d_i,e_i)}|}
$$
such that, for any  $(f_1,\dots,f_{n+m})$ corresponding to a point in $\Omega$,
\begin{itemize}
\item[$(1)$] $Z:= \proj \left( S/(f_1,\dots ,f_{n+m})\right)$ is reduced of dimension zero.
\item[$(2)$] The natural projections of $Z$ to the factors $\mathbb{P}^n$ and $\mathbb{P}^m$  are isomorphism.
\end{itemize}
\end{thm}
\begin{proof}

Let $V:= Z(g_1,\dots,g_{n+m})\subset \mathbb{A}_k^N\times 
\mathbb{P}_k^n\times \mathbb{P}_k^m$ where $\mathbb{A}
_k^N= \spec k[U_{i,\alpha,\beta}]$ for $1\leq i\leq n+m$, $|
\alpha |=d_i$, $|\beta |=e_i$ and  $g_i:= \sum U_{i,\alpha,\beta} \textbf{x}^{\alpha}\textbf{y}^{\beta}$.  Set $\pi: \mathbb{A}_k^N\times 
\mathbb{P}_k^n\times \mathbb{P}_k^m\rightarrow
\mathbb{P}_k^n\times \mathbb{P}_k^m$ be the natural projection, then $V$ is a vector bundle over $
\mathbb{P}_k^n\times \mathbb{P}_k^m$ via $\pi$. For any point $p\in 
\mathbb{P}_k^n\times \mathbb{P}_k^m$ the fiber of $p$ is a linear space of dimension $N-n-m$. Hence $V$ is geometrically irreducible scheme and  $\dim (V)= n+m+(N-n-m)= N$.  \\

$(1)$ Set $W:= (g_1,\dots,g_{n+m}, \Jac^{n+m}(g_1,\dots,g_{n+m}))\subset V$ where $\Jac^{n+m}$ is the Jacobian of order $n+m$. By using Proposition \ref{prop:Bertini} inductively, there exist a point $u\in \mathbb{A}_k^N$ and $(f_1,\dots,f_{n+m})$ corresponding to $u$ such that $\Jac^{n+m}(f_1,\dots,f_{n+m})\neq 0$ which yields $W\subsetneq V$ and  $\dim (W)< \dim (V)=N$. Consider the natural projection $p: \mathbb{A}_k^N\times 
\mathbb{P}_k^n\times \mathbb{P}_k^m\rightarrow  \mathbb{A}_k^N$. Since $\dim (W)<N$, $p(W)$ is a non empty closed subset of $\mathbb{A}_k^N$. Set $\Omega_0 := \left(\mathbb{A}_k^N \setminus \pi(W)\right)$. Therefore for any $(f_1,\dots,f_{n+m})$ corresponding  to a point in $\Omega_0$, $\proj \left( S/(f_1,\dots ,f_{n+m})\right)$ is geometrically reduced and is of dimension zero.\\

$(2)$ Set $p_1: \mathbb{A}_k^N\times 
\mathbb{P}_k^n\times \mathbb{P}_k^m \rightarrow \mathbb{A}_k^N\times 
\mathbb{P}_k^n$ and $q_2: \mathbb{A}_k^N\times 
\mathbb{P}_k^n\rightarrow \mathbb{A}_k^N$ the natural 
projections. First we show that it is enough to show that the restriction  $p_1':V\rightarrow p_1(V)$ is 
birational. Indeed if $p_1'$ is birational then there exists a closed subset 
$Z_1\subsetneq p_1(V)$ such that setting $W_1:=  
(p_1')^{-1}(Z_1)$ then $p_1'|_{V\setminus W_1}:V\setminus 
W_1\rightarrow p_1(V)\setminus Z_1$ is an isomorphism. As $
\dim(Z_1)<N$, $q_2(Z_1)\subsetneq 
\mathbb{A}_k^N$. Set $\Omega_1:= \mathbb{A}_k^N\setminus \left( 
q_2(Z_1)\cup \pi(W)\right)$. For any point $u\in \Omega_1$ and $(f_1,\dots,f_{n
+m})$ corresponding to  $u$,
 $$
q_2^{-1}(u)= p_1(u,Z(f_1,\dots,f_{n+m}))\cong \left(u,Z(f_1,\dots,f_{n+m})\right)
$$
showing that the natural projection of $Z$ onto $\mathbb{P}^n$ is an isomorphism. By replacing $p_1$ and $q_2$ by $p_2: \mathbb{A}_k^N\times 
\mathbb{P}_k^n\times \mathbb{P}_k^m \rightarrow \mathbb{A}_k^N\times 
\mathbb{P}_k^m$ and $q_1: \mathbb{A}_k^N\times 
\mathbb{P}_k^m\rightarrow \mathbb{A}_k^N$, along the same lines it shows the existence of  $Z_2\in p_2(V)$ and  $\Omega_2= \mathbb{A}_k^N\setminus \left( 
q_1(Z_2)\cup \pi(W)\right)$ such that for $u\in \Omega_2$ and $(f_1,\dots,f_{n
+m})$ corresponding to  $u$,
 $$
q_1^{-1}(u)= p_2(u,Z(f_1,\dots,f_{n+m}))\cong \left(u,Z(f_1,\dots,f_{n+m})\right).
$$

Consider the following diagram
$$
\xymatrix{
V\ar^{\subset}[r]\ar^{p_1'}[d]& \mathbb{A}_k^N\times\mathbb{P}_k^n\times \mathbb{P}_k^m\ar^{p_1}[d] \\
V_1:= p_1(V)\ar^{\subset}[r]\ar^{q_2|_{V_1}}[d]
&\mathbb{A}_k^N\times \mathbb{P}_k^n\ar^{q_2}[d]\\
\mathbb{A}_k^N\ar^{=}[r] & \mathbb{A}_k^N.\\
} 
$$
We now show that $p_1'$ is birational.

Let $I:= (g_1,\dots,g_{n+m)}$ and $J:= I:B^{\infty}$. By part $(1)$, $V_1$ is a geometrically  irreducible scheme. Hence $J$ and $I_{V_1}= J\cap k[\textbf{x},\textbf{U}]$ are prime and remain prime under any extension of $k$.

As $V_1$ is generically smooth, there exists $G\in I_{V_1}$ such that $ D:= \dfrac{\partial 
G}{\partial U_{i,\alpha,\beta}}\notin I_{V_1}$ for some $i, \alpha$, and $ \beta$  unless $G\in k[\textbf{y}]
$ which is impossible: indeed $I_{V_1}\cap k[\textbf{y}]=(0)$ because  the projection  $V\rightarrow \mathbb{P}^m$ is 
onto.

Set $Z_1:= Z\left( D\right)\subset V_1$. By \cite[Lemma 4.6.1]{Jouanolou},  
$$
\textbf{x}^{\alpha} \textbf{y}^{\beta} \dfrac{\partial 
G}{\partial U_{i,\alpha',\beta'}}- \textbf{x}^{\alpha '}\textbf{y}^{\beta'}\dfrac{\partial 
G}{\partial U_{i,\alpha,\beta}}\in I_{V}
$$
 for any choices of  $\alpha'$ and $\beta'$. Let $\textbf{y}^{\beta}= y_p\textbf{y}^{\beta_{0}}$ for some $i$. For any $0\leq j\leq m$ put $\alpha':= \alpha$ and $\beta':= y_j\beta_0$. As $\textbf{x}^{\alpha}$ and $\textbf{y}^{\beta_0}$ are not in $I_{V}$ it follows that 
 $$
y_p \dfrac{\partial 
G}{\partial U_{i,\alpha',\beta'}}- y_j\dfrac{\partial 
G}{\partial U_{i,\alpha,\beta}}\in I_{V}
 $$
 for all $0\leq j\leq m$. Therefore by localizing at $D$, we get an isomorphism 
 $$
\phi: \dfrac{\left( k[\textbf{U},\textbf{x}]_D[\textbf{y}]\right)}{I_V}\overset{\cong}\longrightarrow \dfrac{\left(k[\textbf{U},\textbf{x}]_D[y_p]\right)}{I_{V_1}}.
 $$
Therefore the natural maps 
$$
\left(\dfrac{ k[\textbf{U},\textbf{x}]}{I_{V_1}}\right)_{D}\overset{\iota}\hookrightarrow \dfrac{\left( k[\textbf{U},\textbf{x}]_D[\textbf{y}]\right)}{I_V}\overset{\phi}\longrightarrow \dfrac{\left(k[\textbf{U},\textbf{x}]_D[y_p]\right)}{I_{V_1}}
 $$
 are such that $\phi\circ \iota$ induces the identity from  $\proj \left( 
\left(k[\textbf{U},\textbf{x}]_D[y_p]\right)/I_{V_1}\right) $ to $\spec \left(  \left(k[\textbf{U},\textbf{x}]_D\right)/I_{V_1}\right)$.  Hence $\phi$ provides the inverse $\varphi$ 
$$
 V_1\setminus Z_1\overset{\varphi}\longrightarrow  V\setminus \left(Z_1\times \mathbb{P}_k^m\right)\overset{p_1'}\longrightarrow V_1\setminus Z_1.
$$ 
\end{proof}
\begin{thm}\label{thm:generic}
Let $S=k[x_0,\dots,x_n,y_0,\dots,y_m]$ be a bigraded polynomial ring over a field $k$ of characteristics zero. If $I$ is generated by $n+m$ generic forms of bidegree $(d,e)$,  then the scheme $V$ defined by $I$ is a set of reduced points and 
$$
| \lbrace \mu\in \mathbb{N}^2 | \HF_V(\mu)\neq \deg(V)\rbrace| < \infty,
$$
which means the natural projections are one-to-one.
\end{thm}
\begin{proof}
By Theorem \ref{projection generic} part $(1)$, $V$ is a set of reduced points and by part $(2)$,   $\HF_V(0,i)= \HF_V(j,0)=\deg(V)$ for any  $i,j\gg 0$.
\end{proof}

\bibliographystyle{acm}
\bibliography{ref.bib}

\end{document}